\documentclass[12pt,reqno]{amsart}
\usepackage{geometry} 
\usepackage{amssymb}
\usepackage{amscd}
\usepackage{amsmath}
\usepackage[all,cmtip]{xy}
\usepackage{mathrsfs}
\usepackage{amsthm}
\usepackage[pdfpagelabels]{hyperref}
\usepackage{multirow}

\newtheorem{thm}{Theorem}[section]

\newtheorem{prop}[thm]{Proposition}
\newtheorem{conj}[thm]{Conjecture}

\theoremstyle{definition}
\newtheorem{defn}[thm]{Definition}

\theoremstyle{remark}
\newtheorem{rem}[thm]{Remark}

\begin{document}

\title[Map-germs of  corank 2 from $n$-space to $(n+1)$-space]{EXAMPLES of FINITELY DETERMINED MAP-GERMS OF \\ CORANK 2  FROM $n$-SPACE TO $(n+1)$-SPACE}

\author{Ay{\c s}e Alt{\i}nta{\c s} Sharland}

\address{Department of Mathematics, Y{\i}ld{\i}z Technical University, Davutpa{\c s}a Campus,  Esenler 34210, {\.I}stanbul, Turkey}
\curraddr{Institute of Mathematical Sciences, Stony Brook University, Stony Brook NY 11794-3660, USA}

\email{aysealtintas@gmail.com}

\keywords{finite determinacy, multiple point spaces, Mond conjecture}

\subjclass[2000]{58K40, 32S30}

\begin{abstract} We produce new examples supporting the Mond conjecture which can be stated as follows. The number of parameters needed for a miniversal unfolding of a finitely determined 	map-germ from $n$-space to $(n+1)$-space is less than (or equal to if the map-germ is weighted homogeneous)  the rank of the $n$th homology group of the image of a stable perturbation of the map-germ.  We give examples of finitely determined map-germs of corank 2 from $3$-space to $4$-space satisfying the conjecture.  We introduce a new type of augmentations to generate
series of finitely determined map-germs in dimensions $(n,n+1)$ from a given one in dimensions $(n-1,n)$. We present more examples in dimensions $(4,5)$ and $(5,6)$ based on our examples, and verify the conjecture for them.
\end{abstract}
\maketitle

\section{Introduction}

One of the intriguing problems in Singularity Theory is to relate algebraic properties of holomorphic map-germs with topological properties of the images of their stable perturbations.
The image of a stabilisation of a finitely $\mathcal{A}$-determined map-germ from $(\mathbb{C}^n,0)$ to $(\mathbb{C}^{n+1},0)$ has the homotopy type of a wedge of $n$-spheres if $(n,n+1)$ is in the range of Mather's nice dimensions, i.e. $n<15$, (\cite{mond89}) -- the existence of stabilisations is not guaranteed outside the nice dimensions; however, a similar statement can be proved for \textit{topological stabilisations} which exist even for $n\geq 15$ (\cite[Section 4]{damon-mond}, \cite{goryunov-mond}). The number of spheres in the wedge is an $\mathcal{A}$-invariant of the map-germ which is called the \textit{image Milnor number} and denoted by $\mu_I$.
Pellikaan and de Jong (unpublished) then  de Jong and van Straten (\cite{jong-straten}) and later Mond (\cite{mond89}) proved the following result resembling the relation between Milnor and Tjurina numbers for isolated complete intersection singularities (\cite{greuel80}, \cite{looijenga-steenbrink}, see also \cite{milnor} for the hypersurface case). For any finitely $\mathcal{A}$-determined map-germ $f$ from a surface to $3$-space, \[\mathcal{A}_e\textnormal{-codim}(f)\leq \mu_I(f)\] and with equality if $f$ is weighted homogeneous where $\mathcal{A}_e\textnormal{-codim}(f)$ is the dimension of the base of a mini-$\mathcal{A}_e$-versal unfolding of $f$. A similar result for map-germs from $(\mathbb{C},0)$ to $(\mathbb{C}^2,0)$ was also proved by Mond (\cite{mond-bent}). Motivated by these results, Mond suggested the following generalisation.
\begin{conj}[\cite{mond89}, \textbf{Mond Conjecture}]\label{conjmond} Let $f\colon (\mathbb{C}^n,0)\rightarrow (\mathbb{C}^{n+1},0)$ be a finitely $\mathcal{A}$-determined map-germ. 
Then \begin{equation}\label{mui}
\mathcal{A}_e\textnormal{-codim}(f)\leq \mu_{I}(f)
\end{equation}  and with equality if $f$ is weighted homogeneous and $n<15$.
\end{conj}

We claim that Conjecture \ref{conjmond} would follow from  the following statement.
\begin{conj}\label{ngcm} Let $F$ be a stable $d$-parameter unfolding of a finitely $\mathcal{A}$-determined map-germ $f\colon (\mathbb{C}^n,0)\rightarrow (\mathbb{C}^{n+1},0)$ with $n<15$, and $H$ be the defining equation of the image of $F$. Assume that
$G\colon (\mathbb{C}^{n+1}\times \mathbb{C},0)\rightarrow (\mathbb{C}^{n+1}\times \mathbb{C}^d,0) $ is a map-germ transverse to $F$ and induces a stabilisation of $f$. Then $N\mathcal{K}_{H,e/\mathbb{C}}G$ is a Cohen-Macaulay module of dimension 1.
\end{conj}

A stabilisation of a map-germ $f\colon (\mathbb{C}^n,0)\rightarrow (\mathbb{C}^p,0)$ is a deformation $f_{t}$ which has a representative, also denoted by $f_t$, defined on a neighbourhood $U$ of $0$ such that for sufficiently small and nonzero $t$, $f_{t}|_U\colon U\rightarrow \mathbb{C}^{p}$ is stable. The module $N\mathcal{K}_{H,e/\mathbb{C}}G$ is the relative normal space of $G$ with respect to $\mathcal{K}_H$-equivalence, a relation induced by an action of diffeomorphisms preserving the level sets of $H$. See Section 2 for more details.

Our claim is motivated by Damon and Mond's work \cite{damon-mond} in which they proved a version of Conjecture \ref{ngcm} to obtain a phenomenon similar to (\ref{mui}):
\begin{equation}\label{quote1}\parbox{.8\linewidth}{ Let $F$ be a stable $d$-parameter unfolding of a finitely $\mathcal{A}$-determined map-germ $f\colon (\mathbb{C}^n,0)\rightarrow (\mathbb{C}^{p},0)$ where $(n,p)$ are nice dimensions and $n\geq p$, and $H$ be the defining equation of the image of $F$. Assume that
$G\colon (\mathbb{C}^{p}\times \mathbb{C},0)\rightarrow (\mathbb{C}^{p}\times \mathbb{C}^d,0) $ is a map-germ transverse to $F$ and induces a stabilisation of $f$. Then $N\mathcal{K}_{H,e/\mathbb{C}}G$ is a Cohen-Macaulay module of dimension 1.}\end{equation}

\noindent For a finitely $\mathcal{A}$-determined map-germ in these dimensions,  the discriminant of a stabilisation intersected with a Milnor ball $B_\epsilon$ about the origin has the homotopy type of a wedge of $(p-1)$-dimensional spheres (\cite[Theorem 4.6]{damon-mond}). The number $\mu_{\Delta}$ of spheres in the wedge is called the \textit{discriminant Milnor number}. As a consequence of (\ref{quote1}),
\begin{equation}\label{mudelta} \mathcal{A}_e\textnormal{-codim}(f)\leq \mu_{\Delta}(f)\end{equation}
and with equality if $f$ is weighted homogeneous.

Here we run through the main steps of the proof of the inequality (\ref{mudelta}) to clarify the equivalence of the two conjectures above. The proof is mostly based on the relations between $\mathcal{A}$, $\mathcal{K}_V$ and $\mathcal{K}_H$-equivalences where $\mathcal{K}_V$ and $\mathcal{K}_H$ are subgroups of the contact group  $\mathcal{K}$ and measure the $\mathcal{K}$-equivalence of the germs by an equivalence which preserves $V$ and all level sets of $H$, respectively.   

Consider $f\colon (\mathbb{C}^n,0)\rightarrow (\mathbb{C}^{p},0)$ as a pullback of a stable $d$-parameter unfolding $F$ by an immersion $g\colon (\mathbb{C}^{p},0)\rightarrow (\mathbb{C}^{p}\times \mathbb{C}^d,0) $ transverse to $F$. Assume that the discriminant $V$ of $F$ is defined by some $H\in \mathcal{O}_{\mathbb{C}^{p+d},0}$.  For any 1-parameter deformation $G(\textbf{y},t)=g_t(\textbf{y})$ of $g$ that is transverse to $F$ and induces a stabilisation of $f$,  $N\mathcal{K}_{H,e/\mathbb{C}}G$  is a Cohen-Macaulay module of dimension 1 (\cite[Proposition 5.2]{damon-mond}). The proof of this result essentially depends on the fact that $V$ is a free divisor (which is no longer true if $p=n+1$).  Consequently, $\mathcal{K}_H$-equivalence has a \textit{free deformation theory}; that is,
\[\textnormal{dim}_\mathbb{C} N\mathcal{K}_{H,e}g=\sum_{(\textbf{y},t)\in \textnormal{Supp}(N\mathcal{K}_{H,e/\mathbb{C}}G)} \textnormal{dim}_\mathbb{C}  \left(N\mathcal{K}_{H,e}g_t\right)_{\textbf{y}}\] for sufficiently small $t\in \mathbb{C}$ such that $0$, the isolated $\mathcal{K}_H$-instability point of $g$, splits into finitely many $\mathcal{K}_H$-instability points $y_1,\ldots,y_r$ of $g_t$ (\cite[Corollary 5.5]{damon-mond}). 

Let $h_t=H\circ g_t$ so that $h_t$ is a deformation of the image of $f$ and let $J_{h_t}$ be the Jacobian ideal of $h_t$. If $g_t(\textbf{y})\notin V$
\begin{eqnarray*} ev_H\colon (N\mathcal{K}_{H,e}g_t)_{\textbf{y}}&\rightarrow &\mathcal{O}_{\mathbb{C}^p,\textbf{y}}/J_{h_t} \\
\sum_{i=1}^{p+d}\alpha_i \frac{\partial}{\partial Y_i} &\mapsto & \sum_{i=1}^{p+d}\alpha_i \frac{\partial H}{\partial Y_i}\circ g_t \textnormal{ mod }J_{h_t}\end{eqnarray*}
is an isomorphism (\cite[Lemma 5.6]{damon-mond}). Let $\mu(h_t;\mathbf{y}):=\textnormal{dim}_{\mathbb{C}}\mathcal{O}_{\mathbb{C}^p,\textbf{y}}/J_{h_t}$. So, 
\begin{equation}\label{damond-eq}\textnormal{dim}_\mathbb{C} N\mathcal{K}_{H,e}g=\sum_{g_t(\textbf{y})\in V} \textnormal{dim}_\mathbb{C}  \left(N\mathcal{K}_{H,e}g_t\right)_{\textbf{y}}+ \sum_{g_t(\textbf{y})\notin V}\mu(h_t;\mathbf{y}).\end{equation}
Let $B$ be a neighbourhood of $0\in \mathbb{C}^p\times\mathbb{C}^d$ and $t$ sufficiently small so that \\$\textnormal{Supp}(N\mathcal{K}_{H,e/\mathbb{C}}G)\cap (\mathbb{C}^p\times\{t\})\subseteq g_t^{-1}(B)$. In that case, $h_t$ is topologically trivial over the Milnor sphere $S_\epsilon$. Then, $V_t:=h_t^{-1}(0)$ is homotopy equivalent to a wedge of  $(p-1)$-dimensional spheres by a theorem of L{\^e} (\cite{le-concept}). Moreover, $\mu_\Delta$ is equal to the sum of the Milnor numbers of the fibres different than $V_t$ (\cite[Theorem 2.3]{siersma}). In other words, $ \sum_{g_t(\textbf{y})\notin V}\mu(h_t;\mathbf{y})=\mu_\Delta$. 

Suppose that as above, $g_t$ induces a stabilisation of $f$. We have $\mathcal{A}_e\textnormal{-codim}(f)=\mathcal{K}_{V,e}\textnormal{-codim}(g)$ (\cite{damon89}, see also \cite[Theorem 8.1]{mond-diff}); moreover, if $H$ is the defining equation of $V$, $\mathcal{K}_{H,e}\textnormal{-codim}(g)\geq \mathcal{K}_{V,e}\textnormal{-codim}(g)$ (both finite at the same time) and with equality if $F$, $g$ are weighted homogeneous for the same positive weights (\cite[Section 3]{damon-mond}). In the nice dimensions, every stable map is equivalent to a weighted homogeneous one. So, $(N\mathcal{K}_{H,e}g_t)_{\textbf{y}}=(N\mathcal{A}_{e}f_t)_{\textbf{y}}=0$  as $f_t$ is stable. Hence, by (\ref{damond-eq}) $\textnormal{dim}_\mathbb{C} N\mathcal{K}_{H,e}g=\mu_\Delta(f)$. Therefore, (\ref{mudelta}) follows from the fact that
$\mathcal{A}_e\textnormal{-codim}(f)=\mathcal{K}_{V,e}\textnormal{-codim}(g)\leq \mathcal{K}_{H,e}\textnormal{-codim}(g)$.

In the case that $p=n+1$, the discriminant coincides with the image. All the statements above can be adopted to this case except for the fact that $N\mathcal{K}_{H,e/\mathbb{C}}G$ is Cohen-Macaulay module of dimension 1 since the image of a stable unfolding is no longer a free divisor. However, there is no evidence that $N\mathcal{K}_{H,e/\mathbb{C}}G$ cannot have that property whence our claim.

A classification of finitely $\mathcal{A}$-determined map-germs can be pursued as an alternative way to attack Conjecture \ref{conjmond}. However, as the dimension and the corank gets higher, classifying such map-germs or even finding examples becomes a difficult task.

In this paper we focus on finding new examples of finitely $\mathcal{A}$-determined map-germs of corank $\geq 2$ in order to test Conjecture \ref{conjmond}.  It is not our aim to obtain a full classification.  In \cite{houston-kirk}, Houston and Kirk gave a classification of  map-germs of corank 1 from $(\mathbb{C}^3,0)$ to $(\mathbb{C}^{4},0)$ with the strata of $\mathcal{A}_e$-codimension $\leq 4$, and showed that all the examples in their list satisfy the conjecture.
We start with the next interesting case: finitely $\mathcal{A}$-determined map-germs of corank 2 in the dimensions $(3,4)$. In Section \ref{geometric}, we give a list of examples which were found by experiment, namely by building finite weighted homogeneous map-germs lying over the classes of 2-jets calculated in \cite[Appendix C]{altintas}.

In Section \ref{examples}, we study finite $\mathcal{A}$-determinacy of 1-parameter unfoldings defined  by a base change operation on stable unfoldings which we will call augmentations. In \cite{cooper}, Cooper defined the \emph{augmentation} of a finite map-germ $f\colon (\mathbb{C}^n,0)\rightarrow (\mathbb{C}^p,0)$  of $\mathcal{A}_e$-codimension 1 to be $A_F\colon (\textbf{x},\lambda)\mapsto (\tilde{F}(\lambda^2,\textbf{x}),\lambda)$ where $F(\textbf{x},u):=(\tilde{F}(\textbf{x},u),u)$ is an $\mathcal{A}_e$-versal unfolding of $f$ and $u\in\mathbb{C}$ ($n<p$). In \cite{houston-aug}, Houston considered augmentations given by
$A_{F,h}(f)\colon (\textbf{x},\textbf{z})\mapsto (\tilde{F}(\textbf{x},h(\textbf{z})),\textbf{z})$ where $h\colon (\mathbb{C}^d,0)\rightarrow (\mathbb{C},0)$ is a germ of a holomorphic function and $F\colon (\mathbb{C}^n\times \mathbb{C},0)\rightarrow (\mathbb{C}^p\times \mathbb{C},0)$ need not be $\mathcal{A}_e$-versal.
In this paper, we introduce a different type of augmentations. We let $F$ be a $d$-parameter unfolding of $f$ with $F(\textbf{x},\textbf{u})=(\tilde{F}(\textbf{x},\textbf{u}),\textbf{u})$ and
$\gamma\colon w\mapsto \gamma(w)\in (\mathbb{C}^d,0)$ a germ of a curve. Then we define the \emph{augmentation} of $f$ by $F$ and $\gamma$ to be the map-germ
$A_{F,\gamma}(f)\colon (\textbf{x},w)\mapsto (\tilde{F}(\textbf{x},\gamma(w)),w)$. Our definition coincides with Houston's for $d=1$ and Cooper's for $\gamma(w)=w^2$.

Suppose that $f\colon (\mathbb{C}^n,0)\rightarrow (\mathbb{C}^{n+1},0)$ is finitely $\mathcal{A}$-determined, $F$ is a stable $d$-parameter unfolding with the image hypersurface $V$. We prove that an augmentation (in our sense) is finitely $\mathcal{A}$-determined if $\gamma$ intersects the $\mathcal{K}_V$-\textit{discriminant} of the identity map $G$ on $(\mathbb{C}^p\times\mathbb{C}^d,0)$, which is the projection of the support of $N\mathcal{K}_{H,e/\mathbb{C}^d}G$ in the parameter space, only at the origin. We present \textit{series} of finitely $\mathcal{A}$-determined map-germs of corank 2 supporting the conjecture in dimensions $(4,5)$ and $(5,6)$. We show that all the map-germs but two in Houston and Kirk's list are augmentations of finitely $\mathcal{A}$-determined map-germs in Mond's classification in \cite{mondclass}.

We use \textsc{Singular} (\cite{singular}) for our calculations on a computer with a 2.4 GHz Intel Core 2 Duo processor and 4 GB memory. The same calculations can be done using other computer algebra programs such as \textsc{Macaulay2} (\cite{macaulay}).

\section{Terminology and Notations}

Our terminology is standard, but the details can be found in \cite{wall} or \cite{martinet}. We denote the space of holomorphic map-germs $f\colon (\mathbb{C}^n,0)\rightarrow (\mathbb{C}^p,0)$ by $\mathcal{E}_{n,p}^0$. The group $\mathcal{A}:=\textnormal{Diff}(\mathbb{C}^n,0)\times \textnormal{Diff}(\mathbb{C}^p,0)$ of local diffeomorphisms  acts on $\mathcal{E}_{n,p}^0$  by $(\phi,\psi)\cdot f\mapsto \psi \circ f \circ \phi^{-1}$ for all $(\phi,\psi)\in\mathcal{A}$.
We say that $f,g\in\mathcal{E}_{n,p}^0$ are $\mathcal{A}$-\textit{equivalent} if $g\in \mathcal{A}\cdot f$. A map-germ $f\in \mathcal{E}_{n,p}^0$ is $\ell$-$\mathcal{A}$-\textit{determined} if every map-germ $g\in\mathcal{E}_{n,p}^0$ with the same $\ell$-jet (at $0$) as $f$ is $\mathcal{A}$-equivalent to $f$. Furthermore, $f$ is  \textit{finitely $\mathcal{A}$-determined} (or $\mathcal{A}$-\textit{finite}) if it is $\ell$-$\mathcal{A}$-determined for some $\ell<\infty$. A map-germ is $\mathcal{A}$-\textit{stable} if any of its unfoldings is $\mathcal{A}$-equivalent to the trivial unfolding $f\times 1$. By fundamental results of Mather, finite determinacy is equivalent to the finite dimensionality of $N\mathcal{A}_ef:=f^*(\Theta_{\mathbb{C}^p,0})/tf(\Theta_{\mathbb{C}^n,0})+f^{-1}(\Theta_{\mathbb{C}^p,0})$, and thus (if $f$ is not stable) to $0\in\mathbb{C}^p$ being an isolated point of instability of $f$. We set  $\mathcal{A}_e\textnormal{-codim}(f):=\textnormal{dim}_{\mathbb{C}}N\mathcal{A}_ef$.

Following \cite{damon87}, for $g\in \mathcal{E}_{s,p}^0$, $(V,0)\subseteq (\mathbb{C}^p,0)$ and $H\in \mathcal{O}_{\mathbb{C}^p,0}$, the normal spaces with respect to $\mathcal{K}_V$ and $\mathcal{K}_H$-equivalences are given respectively by
\begin{align*} N\mathcal{K}_{V,e}g&:=g^*(\Theta_{\mathbb{C}^p,0})/tg(\Theta_{\mathbb{C}^s,0})+g^*\mbox{Der}(\textnormal{-log } V)
\intertext{and }
N\mathcal{K}_{H,e}g&:=g^*(\Theta_{\mathbb{C}^p,0})/tg(\Theta_{\mathbb{C}^s,0})+g^*\mbox{Der}(\textnormal{-log } H)
\end{align*} where $\mbox{Der}(\textnormal{-log } V)$  (resp. $\mbox{Der}(\textnormal{-log } H)$) is the module of logarithmic vector fields tangent to $V$ (resp. to the level sets of $H$). To be more precise,
\begin{align*} \mbox{Der}(\textnormal{-log } V)&:=\left \{ \xi \in \Theta_{\mathbb{C}^p,0} \mid \xi (I(V))\subseteq I(V) \right \}
\intertext{and }
\mbox{Der}(\textnormal{-log } H)&:=\left\{\xi \in \Theta_{\mathbb{C}^p,0} \mid \xi (H)=0 \right\}
\end{align*}
in which $I(V)$ is the ideal of germs vanishing on $V$. Finite $\mathcal{K}_V$ (resp. $\mathcal{K}_H$) determinacy of $g$ is equivalent to the finite dimensionality of $N\mathcal{K}_{V,e}g$ (resp. $N\mathcal{K}_{H,e}g$). By \cite[Proposition 2.2]{damon87}, $g$ is finitely $\mathcal{K}_V$-determined if and only if $g$ is \textit{algebraically transverse} to $V$ in a punctured neighbourhood of $0\in\mathbb{C}^s$, that is, 
$$d_\mathbf{x}g(T_{\mathbf{x}}\mathbb{C}^{s})+\textnormal{Der}(\textnormal{-log }V)(g(\mathbf{x}))=T_{g(\mathbf{x})}\mathbb{C}^{p}$$
for all $\mathbf{x}\in U-\{0\}$ where $U$ is a neighbourhood of $0\in \mathbb{C}^s$ and $\textnormal{Der}(\textnormal{-log }V)(g(\mathbf{x}))$ is the vector space spanned by the generators of $\textnormal{Der}(\textnormal{-log }V)$ evaluated at $g(\mathbf{x})$.  

For a $d$-parameter deformation $G\colon (\mathbb{C}^s\times\mathbb{C}^d,0)\rightarrow (\mathbb{C}^p,0)$ of $g\in \mathcal{E}_{s,p}^0$, \textit{the relative normal spaces} with respect to $\mathcal{K}_V$ and $\mathcal{K}_H$-equivalences are defined respectively by
\begin{align*} N\mathcal{K}_{V,e/\mathbb{C}^d}G&:=G^*(\Theta_{\mathbb{C}^p,0})/tG(\theta_{\mathbb{C}^{s}\times\mathbb{C}^d/\mathbb{C}^d})+G^*\textnormal{Der}(\textnormal{-log }V) 
\intertext{and} 
N\mathcal{K}_{H,e/\mathbb{C}^d}G&:=G^*(\Theta_{\mathbb{C}^p,0})/tG(\theta_{\mathbb{C}^{s}\times\mathbb{C}^d/\mathbb{C}^d})+G^*\textnormal{Der}(\textnormal{-log }H).\end{align*}
The $\mathcal{K}_V$-\textit{discriminant} of $G$ is defined to be $D_V(G):=\rho \left(\textnormal{Supp}(N\mathcal{K}_{V,e/\mathbb{C}^d}G)\right)$ where $\rho\colon (\mathbb{C}^s\times\mathbb{C}^d,0)\rightarrow (\mathbb{C}^d,0)$ is the standard projection. 

Note that here $V$ and $H$ are not necessarily related. However, for our purposes, we will take $V$ to be a hypersurface defined by the function germ $H$. Assume that $f\in \mathcal{E}_{n,p}^0$ is finite and equals the pullback of a stable unfolding $F\in \mathcal{E}_{n+d,p+d}^0$ by an immersion $g\in \mathcal{E}_{p,p+d}^0$ transverse to $F$. Let $V$ be the image of $F$ and $H$ its defining equation. Then \begin{equation}\label{nak} N\mathcal{A}_ef\cong N\mathcal{K}_{V,e}g\end{equation} (\cite{damon89}, see also \cite[Theorem 8.1]{mond-diff}). When $F$ and $g$ are weighted homogeneous of the same weights,  $N\mathcal{K}_{V,e}g=N\mathcal{K}_{H,e}g$ (\cite[Lemma 3.4]{damon-mond}).

For a $f\in\mathcal{E}_{n,p}^0$ with $n<p$, the  \textit{corank} is the $\mathbb{C}$-vector space dimension of the kernel of $\textnormal{d}f(0)$. Let $Q(f):=\mathcal{O}_{\mathbb{C}^n,0}/f^*\mathfrak{m}_{\mathbb{C}^{p},0} $
be the \textit{local algebra} and  $q(f):=\textnormal{dim}_{\mathbb{C}}\hskip2pt Q(f)$ the \textit{multiplicity} of $f$.

\section{Examples and non-examples of finitely determined map-germs from $\mathbb{C}^3$ to $\mathbb{C}^4$}\label{geometric}

The map-germs listed in Table \ref{tablea} are of corank 2, weighted homogenous and finitely $\mathcal{A}$-determined. The ones labeled as $\hat{A}_\bullet$ are a part of the series
\begin{equation}\label{AAk}\hat{A}_k\colon (x,y,z) \mapsto  (x,\ y^k+xz+x^{2k-2}y,\ yz,\  z^2+y^{2k-1})\end{equation}
which lie over the $2$-jet $(x,y^2+xz,yz,z^2)$ for $k=2$ and $(x,xz,yz,z^2)$ for $k\geq 3$.
The map-germs $\hat{B}_\bullet$ are a part of the series
\begin{equation}\label{BBk}\hat{B}_{2\ell+1}\colon (x,y,z) \mapsto  (x,\ y^2+xz, \ z^2+\alpha xy,\ y^{2\ell+1}+ y^{2\ell}z+yz^{2\ell}-z^{2\ell+1}) \end{equation} where $\alpha=-1$ for $\ell=1$, $\alpha=\pm 1$ for $\ell=2,3,4$ and $\alpha=1$ otherwise. They lie over the $2$-jet $(x,y^2+xz,z^2+\alpha xy,0)$.

\begin{table}[ht!]
\caption{The first set of examples}
\label{tablea}\begin{tabular}{@{}cccc@{}} 
Label & $\mathcal{A}_e$-codimension & Weights &  Conjecture \ref{conjmond} \\ 
\hline
$\hat{A}_2$ & $18$ & $(1,2,3)$ & \textnormal{True} \\
$\hat{A}_3$ & $186$ & $(1,2,5)$ & \textnormal{True} \\
$\hat{A}_4$ & $844$ & $(1,2,7)$ & \textnormal{True} \\\hline
$\hat{B}_3$ & $33$ & $(1,1,1)$ & \textnormal{True} \\
$\hat{B}_5$ & $252$ & $(1,1,1)$ & \textnormal{True} \\
$\hat{B}_7$ & $837$ & $(1,1,1)$ & \textnormal{True}\\
$\hat{B}_9$ & $1968$ & $(1,1,1)$ & \textnormal{True} \\
$\hat{B}_{11}$ & $3825$ & $(1,1,1)$ & \textnormal{True} \\
$\hat{B}_{13}$ & $6588$ & $(1,1,1)$ & ? \\
$\hat{B}_{15}$ & $10437$ & $(1,1,1)$ & ?\\ \hline 
\end{tabular}\end{table}

We calculate $\mathcal{A}_e$-codimensions on \textsc{Singular} using the identity (\ref{nak}), and observe that they also satisfy Conjecture \ref{ngcm}.
The data are given for $k=2,3,4$ and $\ell=1,\ldots,7$ because we have been able to get results only for those values of $k$ and $\ell$ due to computer memory restrictions.

\begin{prop}\label{noex1} There are no finitely $\mathcal{A}$-determined homogeneous map-germs of corank 2 in $\mathcal{E}_{3,4}^0$ with degrees $(1,2,2,2d)$, $d\geq 1$.
\end{prop}
\begin{proof} A finite homogeneous map-germ of corank $2$ with degrees $(1,2,2,2d)$ is $\mathcal{A}$-equivalent to
\[h\colon (x,y,z)\mapsto \left(x,y^2+axz,z^2+bxy,P_{2d}(x,y,z)\right)\]
where $a,b\in \mathbb{C}$ and $P_{2d}(x,y,z)$ is a homogeneous polynomial of degree $2d$.
Assume that $d=1$. Then,
\[h\sim_\mathcal{A} h'\colon (x,y,z)\mapsto \left(x,y^2+axz,z^2+bxy,cxy+dxz+eyz\right)\]
with $c,d,e\in \mathbb{C}$. A calculation on \textsc{Singular} shows that the \textit{double point scheme} $D^2(h')$ is not an isolated singularity whence $h'$ is not finitely $\mathcal{A}$-determined (see, for example, \cite[Proposition 2.1.13]{altintas}). Now, assume that $d\geq 2$. We will show that $h$ forms a line of quadruple points in the image which will imply that it is not $\mathcal{A}$-finite. Let us consider two lines $L_{+}\colon t\mapsto (0,t,\alpha t)$ and $L_{-}\colon t \mapsto (0,t ,-\alpha t)$ where $\alpha\in\mathbb{C}-\{0\}$. We have
\begin{eqnarray*} h(0,t,\alpha t)&=&(0,t^2,\alpha^2t^2, P_{2d}(0,t,\alpha t)) 
\\
 h(0,t,-\alpha t)&=&(0,t^2,\alpha^2t^2, P_{2d}(0,t,-\alpha t)).
\end{eqnarray*}
Write $P_{2d}(x,y,z)=\sum_{l+m+n=2d}a_{lmn}x^ly^mz^n$ where $a_{lmn}\in \mathbb{C}$, for all $l,m,n$. So,
 \begin{eqnarray} \label{p2deq}
P_{2d}(0,t,\alpha t))=P_{2d}(0,t,-\alpha t)) &\Leftrightarrow &\sum_{m+n=2d} a_{0mn}\alpha^n= \sum_{m+n=2d} a_{0mn}(-\alpha)^n \cr &\Leftrightarrow  & \sum_{\substack{m+2i-1=2d \\ i\geq 1}} 2a_{0mn}\alpha^{2i-1}=0.
\end{eqnarray}
Let $\alpha_1$ be a solution of (\ref{p2deq}). Then
\[h(0,t,\alpha_1 t)=h(0,t,-\alpha_1 t)=h(0,-t,-\alpha_1 t)=h(0,-t,\alpha_1 t)\]  whence the result.
\end{proof}

\begin{rem} Notice that $\mathcal{A}_e$-codimensions of our examples are quite high. The smallest codimension we have encountered so far is $\mathcal{A}_e$-codim$(\hat{A}_2)=18$ among the map-germs of corank 2 in these dimensions. It would be very interesting to see if there exist weighted homogeneous map-germs of lower $\mathcal{A}_e$-codimension in the corank 2 case aside from the stable map-germs. The minimal $\mathcal{A}^2$-codimension, that is, the vector space dimension of the complement of $tf(\Theta_{\mathbb{C}^3,0})+f^{-1}(\Theta_{\mathbb{C}^4,0})+\mathfrak{m}_{\mathbb{C}^3,0}^3\cdot \Theta(f)$ in $\mathfrak{m}_{\mathbb{C}^3,0}^2\cdot \Theta(f)$, among the $2$-jets of corank 2 map-germs in $\mathcal{E}_{3,4}^0$ is $6$ (see \cite[Appendix C]{altintas}). This fact suggests that minimal $\mathcal{A}_e$-codimension should be bigger than $6$.  \end{rem}

\begin{rem} Finitely $\mathcal{A}$-determined map-germs of corank 3 are even harder to find. One could start with producing stable corank 3 map-germs by unfolding a finite map-germ of corank 3 in 3 variables as described by Mather in Lemma 5.9 and Theorem 5.10 of \cite{matherIV}, then study nonlinear sections of singularities defined by those stable map-germs to look for $\mathcal{A}$-finite corank 3 map-germs in lower dimensions. However, new map-germs obtained this way may not be weighted homogenous even if the original map-germ is, as the weights of the unfolding parameters may be higher than the degree of the map-germ.
\end{rem}

\section{A new type of augmentation}\label{examples}

In this section, we introduce a new construction of 1-parameter unfoldings to generate new examples of finitely $\mathcal{A}$-determined map-germs from old.

\begin{defn}[cf.\ XIII, 1.4, \cite{martinet}]
 Let $F\in\mathcal{E}_{n+d,p+d}^0$ be a $d$-parameter unfolding of a finite map-germ $f\in\mathcal{E}_{n,p}^0$ given by $F(\textbf{x},\textbf{u})=(F_{\textbf{u}}(\textbf{x}),\textbf{u})$ for $\textbf{u}\in\mathbb{C}^d$, and
  $\gamma\colon (\mathbb{C},0)\rightarrow (\mathbb{C}^d,0)$ be a holomorphic map-germ. We define the \textit{augmentation} $A_{F,\gamma}(f)$ of $f$ by $F$ and $\gamma$ to be the map-germ
\begin{eqnarray*} A_{F,\gamma}(f)\colon (\mathbb{C}^{n}\times \mathbb{C},0)&\rightarrow &(\mathbb{C}^{p}\times \mathbb{C},0) \\
(\textbf{x},w)&\mapsto & (F_{\gamma(w)}(\textbf{x}),\ w).\end{eqnarray*}
In this setting, we will refer to $f$ as the \textit{initial map-germ} of $A_{F,\gamma}(f)$.
\end{defn}

Notice that $A_{F,\gamma}(f)$ is a 1-parameter unfolding of $f$ and its  corank equals the corank of $f$. Moreover, $Q(f)\cong Q(A_{F,\gamma}(f))$. We also have the pull-back diagrams 
\begin{equation}\label{eqpullback}\xymatrix{(\mathbb{C}^n\times \mathbb{C}^d,0) \ar[r]^{F} & (\mathbb{C}^p\times \mathbb{C}^d,0) \\
(\mathbb{C}^n\times\mathbb{C},0) \ar[r]^{A_{F,\gamma}(f)} \ar[u] & (\mathbb{C}^p\times\mathbb{C},0) \ar[u]_{\hat{g}\colon  (\textbf{Y},w) \mapsto (\textbf{Y},\gamma(w))}\\
(\mathbb{C}^n,0) \ar[r]^{f} \ar[u] & (\mathbb{C}^p,0) \ar[u]_{i\colon \mathbf{Y}\mapsto (\mathbf{Y},0)}}\end{equation} with $\textbf{Y}=(Y_1,\ldots,Y_{p})\in \mathbb{C}^p$ and $w\in \mathbb{C}$.

Now we will state a criterion for finite determinacy of augmentations if $p=n+1$. 

\begin{thm}\label{newex}  Let $F\in \mathcal{E}_{n+d,n+d+1}^0$ be a parametrised stable unfolding of  an $\mathcal{A}$-finite $f\in\mathcal{E}_{n,n+1}^0$ and  $V$ the image of $F$. Let $G$ be the identity on $(\mathbb{C}^{n+1}\times\mathbb{C}^d,0)$. Assume that  $\gamma\colon (\mathbb{C},0) \rightarrow (\mathbb{C}^d,0)$ is a non-constant map-germ parametrising a curve which intersects  $D_V(G)$ only at the origin.  Then, $A_{F,\gamma}(f)$ is also $\mathcal{A}$-finite. \end{thm}

\begin{proof}  Consider  $A_{F,\gamma}(f)$ as a pull-back of $F$ by $\hat{g}$ as in (\ref{eqpullback}). Then, $\hat{g}$ is transverse to $F$ since $g:=\hat{g}\circ i \colon \textbf{Y} \mapsto (\textbf{Y},0)$ is. Consequently, $N\mathcal{A}_{e}A_{F,\gamma}(f)\cong N\mathcal{K}_{V,e}\hat{g}$. Hence, it suffices to show that the support of $N\mathcal{K}_{V,e}\hat{g}$ consists of the origin at most. 

By definition, $D_V(G)$ is the set of points $\textbf{U}\in \mathbb{C}^d$ for which $G(-,\textbf{U})$ fails to be algebraically transverse to $V$ at $(\mathbf{Y},\mathbf{U})$. We have $G(\mathbf{Y},\gamma(w))=\hat{g}(\mathbf{Y},\gamma(w))$. So, by the choice of $\gamma$, $\hat{g}$ is algebraically transverse to $V$ at $(\mathbf{Y},w)$ for any $w\neq 0$. Since $f$ is $\mathcal{A}$-finite, $g$ is algebraically transverse to $V$ off the origin which implies that $\mathbb{C}^{n+1}\times \{0\}$ is transverse to $V$ at all points except $(0,0)$. Hence, the support of $N\mathcal{K}_{V,e}\hat{g}$ intersects $\mathbb{C}^{n+1}\times \{0\}$ only at the origin. This concludes the proof.
\end{proof}

\begin{rem} The converse of Theorem \ref{newex} does not always hold. For instance, consider $\hat{A}_2$ (see Table \ref{tablea})  which is an augmentation of $f\colon (y,z)\mapsto (y^2,yz,z^2+y^3)$ by 
$F(y,z,\mathbf{u})=(y^2+u_1y+u_2z,\ yz,\  z^2+y^{3}+u_3y+u_4z,\ \mathbf{u})$
and $\gamma(x)=(x,x^2,0,0)$. The image of $\gamma$ intersects $D_V(G)$ only at $\mathbf{0}$. However, $f$ is not finitely $\mathcal{A}$-determined. Still, this fact does not suggest that $\hat{A}_2$ cannot be equivalent to an augmentation of an $\mathcal{A}$-finite map-germ (see Remark \ref{remAB}). We leave this problem for further study.
\end{rem}

\subsection{Examples of augmentations in dimensions $(3,4)$}

\begin{rem}\label{remAB} The map-germs listed in Table \ref{tablea} are also augmentations of $\mathcal{A}$-finite map-germs. The series $\hat{A}_k$ given by (\ref{AAk}) is $\mathcal{A}$-equivalent to the map-germ
\begin{equation}\label{AAk2}(x,y,z)\mapsto  (x,\ y^k+xz+(x+yz)^{2k-2}y+yz^2,\ yz,\  z^2+y^{2k-1}).\end{equation}
The initial map-germ of (\ref{AAk2}) is
$$h_k\colon (y,z)\mapsto   (y^k +y^{2k-1}z^{2k-2}+yz^2,\ yz,\  z^2+y^{2k-1}).$$
We see that the projection $D^2_1(h_k)$ of the double point scheme $D^2(h_k)$ into $\mathbb{C}^2$ is an isolated singularity for $k=2,3,4$. Hence $h_k$ is finitely $\mathcal{A}$-determined (see \cite{mond87} for the criteria).  On the other hand, the initial map-germ of $\hat{B}_{2\ell+1}$, given by (\ref{BBk}), is
$$\hat{h}_{2\ell+1} \colon (y,z)\mapsto (y^2, \ z^2,\ y^{2\ell+1}+y^{2\ell}z+yz^{2\ell}-z^{2\ell+1}).$$
Finite $\mathcal{A}$-determinacy of $\hat{h}_1$ was stated in \cite{bruce-marar}. For $\ell\geq 1$, it again follows from the fact that $D^2_1$ is an isolated singularity which can be checked by a calculation. In fact, $\mathcal{A}_e\textnormal{-codim}(\hat{h}_{2\ell+1})=10\ell^2+2\ell$ for all $\ell\geq 1$ by Mond's formula for $\mathcal{A}_e$-codimension of weighted homogeneous map-germs in $\mathcal{E}_{2,3}^0$ (\cite{mond-novanish}). 
\end{rem}

\begin{prop}\label{brucemarar1} The map-germ
\begin{equation}\label{ex-bruce}f_\ell\colon (x,y,z)\mapsto (x,y^2+x^\ell z,\ z^2-x^\ell y,\ y^3+y^2z+yz^2-z^3)\end{equation}
is finitely $\mathcal{A}$-determined and satisfies Conjecture \ref{conjmond} for all $\ell\geq 1$.
\end{prop}

\begin{proof} We observe that $f_\ell$ is an augmentation of 
\[ \hat{h}_1\colon (y,z)\mapsto (y^2,\ z^2,\ y^3+y^2z+yz^2-z^3)\]
by
\[F(y,z,\textbf{u})=(y^2+u_1z,\ z^2+u_2y+u_3z,\  y^3+y^2z+yz^2-z^3+u_4y+u_5z,\  \textbf{u})\]
and $\gamma\colon x\mapsto (x^\ell,-x^\ell,0,0,0)$. Let $G$ be the identity on $(\mathbb{C}^2\times \mathbb{C}^5,0)$. The following calculation on \textsc{Singular} shows that the image of $\gamma$ cuts $D_V(G)$ only at $\{0\}$.  

\vskip7pt
\noindent\verb|LIB "matrix.lib";|\\
\verb|ring T=0,(Y,Z,W,U1,U2,U3,U4,U5),(wp(1,4,5,6,2,3,4));|\\
\verb|ring S=0,(y,z,u1,u2,u3,u4,u5),(wp(1,2,3,2,3,4));|\\
\verb|ideal p=0;|\\
\verb|map F=T,y2+u1*z,z2+u2*y+u3*z,y3+y2z+yz2-z3+u4*y+u5*z,u1,u2,u3,u4,u5;|\\
\verb|setring T;|\\
\verb|ideal H=preimage(S,F,p);|\\ 
\verb|ideal jh=jacob(H);|\\
\verb|module derv=modulo(jh,H);| \quad  \quad \quad \quad \hskip18pt  $\setminus \setminus$ $\verb|derv|:=\textnormal{Der}(\textnormal{-log }V)$\\
\verb|def tkv=submat(derv,4..8,1..ncols(derv));|   \quad $\setminus \setminus$ $\verb|tkv|:=t\rho(\textnormal{Der}(\textnormal{-log }V))$\\
\verb|ideal sup=std(minor(tkv,5));|    \quad    \hskip32pt    $\setminus \setminus$ annihilator of $N\mathcal{K}_{V,e/\mathbb{C}^5}G$\\
\verb|ideal ID=eliminate(sup,YZW);|          \quad   \hskip32pt  $\setminus \setminus$ ideal of $D_V(G)$ \\
\noindent \verb|ideal intersect=ID+(U1+U2,U3,U4,U5);|  $\setminus \setminus$ intersection of $D_V(G)$ and $\textnormal{im}(\gamma)$ \\
\verb|radical(intersect);| \\

\noindent Hence, finite $\mathcal{A}$-determinacy follows from Theorem \ref{newex}. In order to show that $f_\ell$ satisfies the conjecture, we consider the 1-parameter deformation 
$$\hat{G}_\ell\colon (Y,Z,W,x,t)\mapsto (Y,Z,W,t+x^\ell, -t-x^\ell, t,t^2,0)$$ of $\hat{g}_\ell$ which induces $f_\ell$ from $F$. Notice that 
$$N\mathcal{K}_{V,e/\mathbb{C}}\hat{G}_\ell/ (t) N\mathcal{K}_{V,e/\mathbb{C}}\hat{G}_\ell  \cong N\mathcal{K}_{V,e}\hat{g}_\ell.$$
As $\textnormal{dim}N\mathcal{K}_{V,e}\hat{g}_\ell=0$, we have $\textnormal{dim} N\mathcal{K}_{V,e/\mathbb{C}}\hat{G}_\ell \leq 1$. On the other hand, for $\ell\geq 2$, 
$$N\mathcal{K}_{V,e/\mathbb{C}}\hat{G}_\ell/ (x) N\mathcal{K}_{V,e/\mathbb{C}}\hat{G}_\ell  \cong N\mathcal{K}_{V,e/\mathbb{C}}G$$
where $G(Y,Z,W,t):=G(Y,Z,W,0,t)$.  By a calculation on \textsc{Singular}, we find that $N\mathcal{K}_{V,e/\mathbb{C}}G$ has dimension $1$. Consequently, $\textnormal{dim} N\mathcal{K}_{V,e/\mathbb{C}}\hat{G}_\ell \geq 1$. Therefore, we must have $\textnormal{dim} N\mathcal{K}_{V,e/\mathbb{C}}\hat{G}_\ell = 1$. Thus, $t$ is a $N\mathcal{K}_{V,e/\mathbb{C}}\hat{G}_\ell$-regular element and $\hat{G}_\ell$ induces a stabilisation of $f_\ell$ for $t\neq 0$ and $\ell\geq 2$. The statement for $\ell=1$ can be checked by a direct calculation. Similarly, one can show that $t$ is also a $N\mathcal{K}_{H,e/\mathbb{C}}\hat{G}_\ell$-regular element. Hence, $N\mathcal{K}_{H,e/\mathbb{C}}\hat{G}_\ell$ is a free module over $\mathcal{O}_{\mathbb{C},0}$. This concludes the proof. 
\end{proof}

\subsection{Examples of augmentations in dimensions $(4,5)$}

\begin{prop}\label{CCk} Each map-germ in the series
\begin{eqnarray*}
\hat{C}_\ell\colon (\mathbb{C}^4,0)&\rightarrow &(\mathbb{C}^5,0) \\
(x,y,z,w)&\mapsto & (x,\ y^2+xz+x^2y,\ yz+w^\ell y,\ z^2+y^3,\ w)
\end{eqnarray*}
is $\mathcal{A}$-finite with $\mathcal{A}$-codimension $30\ell-18$ and satisfies Conjecture \ref{conjmond}  for all $\ell\geq 1$.
\end{prop}

\begin{proof} The map-germ $\hat{C}_\ell$ is an augmentation of $\hat{A}_2$ (see Table \ref{tablea}) by
\[F_{\hat{A}_2}(x,y,z,\textbf{u})=(x,\ y^2+xz+x^2y+u_1y,\ yz+u_2y,\ z^2+y^3+u_3y,\ \textbf{u})\] and $\gamma_{\hat{C}}\colon w\mapsto (0,w^\ell,0)$.

Let $G$ be the identity on $(\mathbb{C}^4\times \mathbb{C}^3,0)$. Any curve $\gamma\colon w\mapsto (0,\gamma_2(w),0)$, where $\gamma_2(w)$ is not constant, intersects $D_V(G)$ only at the origin. Hence, $\hat{C}_\ell$ is also finitely $\mathcal{A}$-determined by Theorem \ref{newex}. 

Now we will show that $N\mathcal{K}_{V,e}\hat{g}_\ell$ has dimension $30\ell-18$ over $\mathbb{C}$ where $\hat{g}_\ell\colon$ $(X,Y,Z,W,w)\mapsto (X,Y,Z,W,0,w^\ell,0)$ for all $\ell\geq 1$.
Let $\textbf{U}=(U_1,U_2,U_3)$ denote the chosen coordinate system on the parameter space.
Clearly, we can represent any $\xi \in \Theta(\hat{g}_\ell)$ as $\xi=(\xi_1,\xi_2)$ where $\xi_1$ is the $\textbf{Y}$-component and $\xi_2$ the $\textbf{U}$-component. So, let $\rho\colon (\mathbf{Y},\mathbf{U})\mapsto \mathbf{U}$ be the standard projection and $\Theta(\hat{g}_\ell)/\rho:=\left\{ \xi_2 \mid \xi=(\xi_1,\xi_2) \in \Theta(\hat{g}_\ell) \right\}$ which is naturally isomorphic to $(\mathcal{O}_{\mathbb{C}^{5},0})^3$. We have
\begin{eqnarray*}
N\mathcal{K}_{V,e}\hat{g}_\ell\cong
\frac{\Theta(\hat{g}_\ell)/\rho}{\big(\frac{\partial \gamma_{\hat{C}}}{\partial w} \big) \mathcal{O}_{\mathbb{C}^5,0}+\hat{g}_\ell^*t\rho \left(\textnormal{Der}(\textnormal{-log }V)\right)}.
\end{eqnarray*}
Let $\frac{\partial\ }{\partial U_1}, \frac{\partial\ }{\partial U_2}, \frac{\partial\ }{\partial U_3}$ denote the standard basis for $\Theta(\hat{g}_\ell)/\rho$. We find that
a Groebner basis for $M_0:=\hat{g}_1^*t\rho \left(\textnormal{Der}(\textnormal{-log }V)\right)$  with respect to the reverse lexicographic order with priority given to the coefficients is given by the following vector fields.
\begin{eqnarray*}
m_1&:=&X^8 \frac{\partial}{\partial U_1}, \hskip12pt 
m_2:= X^9\frac{\partial}{\partial U_2}, 
\hskip12pt 
m_3:= -\frac{277}{229}X^7\frac{\partial}{\partial U_1} -\frac{25}{458}X^8\frac{\partial}{\partial U_2}+ X^9\frac{\partial}{\partial U_3},
\cr\cr
m_4&:=& \frac{5}{2}X^2\frac{\partial}{\partial U_1}-\frac{1}{4}X^3\frac{\partial}{\partial U_2}+\big (Y-\frac{1}{2}X^4\big )\frac{\partial}{\partial U_3},
\cr \cr
m_5&:=& \big (X^2Y+\frac{74}{39}X^6\big )\frac{\partial}{\partial U_1} -\frac{61}{312}X^7\frac{\partial}{\partial U_2} -\frac{59}{156}X^8 \frac{\partial}{\partial U_3},
\cr \cr
m_6&:=& \big (-\frac{8}{7}XY-\frac{36}{7}X^5\big )\frac{\partial}{\partial U_1}+\big ( X^2Y-\frac{4}{7}X^6\big )\frac{\partial}{\partial U_2}+ \frac{40}{7}X^7 \frac{\partial}{\partial U_3}, \cr \cr
m_7&:=& \big (Y^2+\frac{3439}{32}X^4Y\big )\frac{\partial}{\partial U_1} -\frac{19473}{256}X^5Y \frac{\partial}{\partial U_2} -\frac{82731}{128}X^6Y \frac{\partial}{\partial U_3},
\cr \cr
m_8&:=&\big ( \frac{147}{2}X^3Y-\frac{12939}{16}X^7\big )\frac{\partial}{\partial U_1}+\big (Y^2-\frac{817}{16}X^4Y+\frac{4331}{32}X^8\big )\frac{\partial}{\partial U_2}+ \cr&& -\big (\frac{3483}{8}X^5Y+\frac{1173}{16}X^9\big )\frac{\partial}{\partial U_3}, \cr \cr
m_9&:=& \big (Z+\frac{29}{2}XY-\frac{2549}{16}X^5\big )\frac{\partial}{\partial U_1}+\big (-\frac{159}{16}X^2Y+\frac{861}{32}X^6\big )\frac{\partial}{\partial U_2}+ \cr &&+\big (-\frac{693}{8}X^3Y-\frac{243}{16}X^7\big )\frac{\partial}{\partial U_3}, \cr
\cr
m_{10}&:=&\big (2Y+\frac{29}{4}X^4\big )\frac{\partial}{\partial U_1}+\big (Z+\frac{7}{4}XY-\frac{1}{8}X^5\big )\frac{\partial}{\partial U_2}-\big (\frac{3}{2}X^2Y+\frac{9}{4}X^6\big ) \frac{\partial}{\partial U_3}, \cr \cr
m_{11}&:=& \frac{17}{2}X^3\frac{\partial}{\partial U_1}+\big (\frac{1}{2}Y-\frac{3}{4}X^4\big )\frac{\partial}{\partial U_2}+\big (Z+3XY-\frac{3}{2}X^5\big )\frac{\partial}{\partial U_3}, \cr \cr
m_{12}&:=&\big ( W+\frac{189}{40}XZ+\frac{201}{55}X^2Y-\frac{4853}{80}X^6\big )\frac{\partial}{\partial U_1}+\cr &&+\big (\frac{927}{110}X^7-\frac{15}{176}X^2Z-\frac{177}{44}X^3Y\big )\frac{\partial}{\partial U_2}-\big (\frac{4491}{880}X^3Z+\frac{3087}{110}X^4Y\big )\frac{\partial}{\partial U_3}, \cr \cr
m_{13}&:=&\big (-Z+\frac{1677}{88}XY-\frac{18877}{32}X^5\big )\frac{\partial}{\partial U_1}-\big (\frac{11609}{352}X^2Z+\frac{5919}{22}X^3Y\big )\frac{\partial}{\partial U_3}+\cr &&+\big (W+\frac{305}{88}XZ-\frac{3215}{88}X^2Y+\frac{3623}{44}X^6\big )\frac{\partial}{\partial U_2}, \cr \cr
m_{14}&:=&\big (\frac{7}{33}Y+\frac{55}{24}X^4\big )\frac{\partial}{\partial U_1}+\big (\frac{20}{33}Z-\frac{19}{132}XY-\frac{13}{33}X^5\big )\frac{\partial}{\partial U_2}+\cr &&+\big (W-\frac{47}{88}XZ+\frac{41}{22}X^2Y\big )\frac{\partial}{\partial U_3}, \cr \cr
 m_{15}&:=&\big (w+\frac{57}{4}X^3\big )\frac{\partial}{\partial U_1}+\big (Y-2X^4\big )\frac{\partial}{\partial U_2}+\big (\frac{3}{4}Z+6XY\big )\frac{\partial}{\partial U_3}, \cr \cr
m_{16}&:=& w\frac{\partial}{\partial U_2}, \hskip12pt 
m_{17}:= -\frac{3}{2}X\frac{\partial}{\partial U_1}-\frac{1}{4}X^2\frac{\partial}{\partial U_2}+\big ( w+\frac{3}{2}X^3\big )\frac{\partial}{\partial U_3}. \nonumber
\end{eqnarray*}
Let $\iota\colon \left (X,Y,Z,W,w\right )\mapsto \left (X,Y,Z,W,w^\ell\right )$  so that $\hat{g}_\ell=\iota^*\hat{g}_1$.
By an application of Buchberger's Algorithm, we find that  $m_1,\ldots,m_{14},\iota^*m_{15}$, $\iota^*m_{17}$ together with the following  elements form a Groebner basis for $M:=\big(\frac{\partial \gamma_{\hat{C}}}{\partial w} \big ) \mathcal{O}_{\mathbb{C}^5,0}+\hat{g}_\ell^*t\rho\left( \textnormal{Der}(\textnormal{-log }V )\right )$. 
\begin{eqnarray*}
m_{18}&:=& w^{\ell-1}\frac{\partial}{\partial U_2}, \hskip12pt 
m_{19}:= \big (Yw^{\ell-1}+\frac{11}{2}X^4w^{\ell-1}\big )\frac{\partial}{\partial U_1} -\frac{3}{2}X^6w^\ell\frac{\partial}{\partial U_3}, \cr \cr
m_{20}&:=& \frac{2}{7}X^5w^{\ell-1}\frac{\partial}{\partial U_1}+X^7w^{\ell-1}\frac{\partial}{\partial U_3}, \hskip12pt  
m_{21}:= X^5w^{\ell-1}\frac{\partial}{\partial U_1}.
\end{eqnarray*}

The $\mathbb{C}$-vector space dimension of $(\mathcal{O}_{\mathbb{C}^5,0})^3/M$ equals the number of monomials in the complement of the module $\textnormal{LT}(M)$ which is generated by the leading monomials with respect to the chosen ordering (\cite[\S 2, Chapter 5]{clo-using}). Namely,
\begin{equation}
\textnormal{dim}_{\mathbb{C}} N\mathcal{K}_{V,e}\hat{g}_\ell   =  \textnormal{dim}_{\mathbb{C}}\frac{ (\mathcal{O}_{\mathbb{C}^5,0})^3}{\left(\textnormal{LT}(M)\right)\mathcal{O}_{\mathbb{C}^5,0}}.\end{equation} Therefore,
\begin{eqnarray*}
\textnormal{dim}_{\mathbb{C}} N\mathcal{K}_{V,e}\hat{g}_\ell
& = &  \textnormal{dim}_{\mathbb{C}} \frac{\mathcal{O}_{\mathbb{C}^5,0}}{\left(X^8,X^2Y,Y^2, Z,W,X^5w^{\ell-1},Yw^{\ell-1},w^\ell\right)\mathcal{O}_{\mathbb{C}^5,0}} + \nonumber \\ & & +
 \textnormal{dim}_{\mathbb{C}} \frac{\mathcal{O}_{\mathbb{C}^5,0}}{\left(X^9,X^2Y,Y^2, Z,W,w^{\ell-1}\right)\mathcal{O}_{\mathbb{C}^5,0}} + \\ && +
 \textnormal{dim}_{\mathbb{C}} \frac{\mathcal{O}_{\mathbb{C}^5,0}}{\left(X^9,Y, Z,W,X^7w^{\ell-1},w^\ell\right)\mathcal{O}_{\mathbb{C}^5,0}} \\
&  =& (10\ell-5)+(11\ell-11)+(9\ell-2) \\
&  = &30\ell-18. 
 \end{eqnarray*} 
 
 Finally, the fact that $\hat{C}_\ell$ satisfies the conjecture can be checked using the deformation $\hat{G}_\ell\colon (X,Y,Z,W,w,t)\mapsto (X,Y,Z,W,t,w^\ell,t^2)$ of $\hat{g}_\ell$ (cf. the proof of Proposition \ref{brucemarar1}). \end{proof}

By similar calculations, we can prove the following two propositions.
\begin{prop} The map-germ $f_\ell$ defined by (\ref{ex-bruce}) has $\mathcal{A}_e$-codimension $45\ell-12$.
\end{prop}

\begin{prop} The map-germs shown in Table \ref{tableO} are finitely $\mathcal{A}$-determined and satisfy the conjecture for all $\ell\geq 1$. \end{prop}

\begin{table}[!h]
\caption{$\mathcal{A}$-finite map-germs in $\mathcal{E}_{4,5}^0$ and $\mathcal{E}_{5,6}^0$.}
\label{tableO}
\begin{tabular}{@{}cccc@{}}
Label & Augmentation & $\mathcal{A}_e\textnormal{-codim}$ & Unf. \\\hline
$\hat{D}_\ell$ & $(x, y^2+xz+x^2y+w^\ell y,  yz, z^2+y^3+w^{2\ell}y, w)$  & $45\ell$-$18$ & $F_{\hat{A}_2}$ \\ 
$\hat{E}_\ell$ & $(x, y^3+xz+x^4y+w^\ell y, yz, z^2+y^5+w^{2\ell}y, w)$ & $536\ell$-$186$ & $F_{\hat{A}_3}$  
\\
$\hat{J}_\ell$ & $(x, y^3+xz+x^4y+w^{2\ell} y^2, yz+w^{5\ell}y, z^2+y^5, w)$ & $2144\ell$-$186$ & $F_{\hat{A}_3}$  
\\
$\hat{K}_\ell$ & $(x, y^2+xz, z^2+xy, y^3+yz^2+z^3+w^\ell z, w)$ & $51\ell$-$33$& $F_{\hat{B}_3}$ 
\\
$\hat{L}_\ell$ & $(x, y^2+xz, z^2+xy, y^5+y^3z^2+z^5+w^\ell y-w^\ell z, w)$ & $372\ell$-$252$& $F_{\hat{B}_5}$ \\
\hline
$\hat{M}_{1,\ell}$ & $(x, y^2+xz+x^2y+v^\ell y,  yz+wy, z^2+y^3+v^{2\ell}y, w, v)$  & $25\ell$-$12$ & $F_{\hat{C}_1}$\\
$\hat{M}_{2,\ell}$& $(x, y^2+xz+x^2y+v^\ell y,  yz+w^2y, z^2+y^3+v^{2\ell}y, w, v)$  & $95\ell$-$42$ & $F_{\hat{C}_2}$ \\
$\hat{M}_{3,\ell}$& $(x, y^2+xz+x^2y+v^\ell y,  yz+w^3y, z^2+y^3+v^{2\ell}y, w, v)$  & $165\ell$-$72$ & $F_{\hat{C}_3}$ \\
$\hat{N}_{\ell}$ & $(x, y^3+xz+x^2y+wy, yz+v^{\ell}z, z^2+y^5+w^2y+v^{4\ell}y+v^{3\ell}y^2, w, v)$ & $1750\ell$-$350$ & $F_{\hat{E}_1}$\\
$\hat{P}_\ell$ & $(x, y^2+xz+v^\ell z, z^2+xy, y^3+yz^2+z^3+wz+v^{2\ell}y, w, v)$ & $42\ell$-$18$ & $F_{\hat{K}_1}$\\ \hline 
\end{tabular}
\end{table}

See Table \ref{tableunf} for the list of unfoldings from which the series in Table \ref{tableO} are deduced.

\begin{table}[!h]
\caption{Stable unfoldings.}
\label{tableunf}
\begin{tabular}{@{}cc@{}}
Label & Unfolding
 \\  \hline
$F_{\hat{A}_2}$ & $(x,\ y^2+xz+x^2y+u_{11}y,\ yz+u_{12}y,\ z^2+y^3+u_{13}y,\textbf{u}_1)$
\\
$F_{\hat{A}_3}$ & $(x,\ y^3+xz+x^4y+u_{21}y+u_{22}y^2,\ yz+u_{23}y,\ z^2+y^5+u_{24}y+u_{25}y^2,\textbf{u}_2)$
\\
$F_{\hat{B}_3}$ & $(x,\ y^2+xz+u_{31}z,\ z^2+xy+u_{32}z,\ y^3+yz^2+z^3+u_{33}y+u_{34}z, \textbf{u}_3)$
\\
$F_{\hat{B}_5}$ & $(x,\ y^2+xz+u_{41}z,\ z^2+xy,\ y^3+yz^2+z^3+u_{42}y+u_{43}z+u_{44}yz, \textbf{u}_4)$
\\
$F_{\hat{C}_\ell}$ & $(x,\ y^2+xz+x^2y+u_{51}y,\ yz+w^\ell y+u_{52}y,\ z^2+y^3+u_{53}y,w,\textbf{u}_5)$
\\
$F_{\hat{E}_1}$ & $(x,\ y^3+xz+x^4y+wy,\ yz+u_{61}y+u_{62}z,\ z^2+y^5+u_{63}y+u_{64}y^2,w,\textbf{u}_6)$
\\
$F_{\hat{K}_1}$ & $(x,\ y^2+xz+u_{71}z,\ z^2+xy+u_{72}z,\ y^3+yz^2+z^3+wz+u_{73}y,\ w,\textbf{u}_7)$
\\\hline 
\end{tabular}\end{table}

\begin{rem} The smallest codimension we have got so far is $\mathcal{A}_e\textnormal{-codim}(\hat{C}_1)=12$ in the dimensions $(4,5)$ and $\mathcal{A}_e\textnormal{-codim}(\hat{M}_{1,1})=13$ in $(5,6)$.
\end{rem}

\begin{rem} The map-germs $\hat{M}_{1,\ell}$ and $\hat{M}_{2,\ell}$ are actually parts of the series \begin{equation} \hat{M}_{k,\ell}\colon (x,y,z,w,v)\mapsto (x,\ y^2+xz+x^2y+v^k y,\  yz+w^\ell y,\ z^2+y^3+v^{2k}y,w,v)\end{equation}
which are finitely $\mathcal{A}$-determined with $\mathcal{A}_e\textnormal{-codim}=(70k-30)\ell-45k+18$.\end{rem}

\begin{rem} The formula for $\mathcal{A}_e$-codimension for each germ in Table \ref{tableO} is a linear form in one variable and the constant term is equal to the codimension of the initial map-germ. It would be interesting to see if this holds in general.
\end{rem}

\begin{rem}\label{remHK} All map-germs from Houston and Kirk's list of simple singularities of corank 1 (\cite[Table 1]{houston-kirk}) can be constructed from finitely $\mathcal{A}$-determined map-germs in $\mathcal{E}_{2,3}^0$ by our method. We list them in Table \ref{tableHK} together with their initial map-germs. The labels for the initial map-germs are from Mond's list in \cite{mondclass}.
\end{rem}

\begin{table}[!h]
\caption{Houston and Kirk's list of simple map-germs in $\mathcal{E}_{3,4}^0$.}
\label{tableHK}
\begin{tabular}{@{}cccc@{}} 
Label & Map-germ & Initial map & Mond's ref. \\ \hline
$A_k$& $(x,y,z^2, z^3\pm x^2z\pm y^{k+1}z)$  & $(x,z^2,z^3\pm x^2z)$ & $S_1$ \\
$D_k$& $(x,y,z^2, z^3+x^2yz\pm y^{k-1}z)$  & $(y,z^2,z^3\pm y^{k-1}z)$ & $S_{k-2}$ \\ \hline
\multirow{3}{*}{$E_6$} & \multirow{3}{*}{$(x,y,z^2, z^3+x^3z\pm y^{4}z)$}  &  $(x,z^2,z^3+x^3z)$ & $S_2$ \\
& & or & \\
& & $(y,z^2,z^3\pm y^4z)$ & $S_3$ \\ \hline
$E_7$& $(x,y,z^2, z^3+x^3z+xy^{3}z)$
& $(x,z^2,z^3+x^3z)$ & $S_2$ \\ \hline
\multirow{3}{*}{$E_8$} & \multirow{3}{*}{$(x,y,z^2, z^3+x^3z\pm y^{5}z)$}  &  $(x,z^2,z^3+x^3z)$ & $S_2$ \\
& & or & \\
& & $(y,z^2,z^3\pm y^5z)$ & $S_4$ \\\hline
$B_k$& $(x,y,z^2, x^2z\pm y^2z\pm z^{2k+1})$  & $(x,z^2,x^2z\pm z^{2k+1})$ & $B^\pm_{k}$ \\
$C_k$& $(x,y,z^2, x^2z+yz^3\pm y^{k}z)$  & $(y,z^2,yz^3\pm y^kz)$ & $C^\pm_{k}$ \\\hline
\multirow{3}{*}{$F_4$}& \multirow{3}{*}{$(x,y,z^2, x^2z+y^3z\pm z^5)$}  & $(x,z^2,x^2z\pm z^5)$ & $B^\pm_{2}$ \\
& &  or & \\
& & $(y,z^2,y^3z\pm z^5)$ & $F_4$ \\\hline
$P_1$ & $(x,y,yz+z^4,xz+z^3)$  &  $(x,z^4,xz+z^3)$ & $T_4$ \\
$P_2$ & $(x,y,yz+z^5,xz+z^3)$  &  $(y,yz+z^5,z^3)$ & $H_2$ \\
$P_3^k$ & $(x,y,yz+z^6\pm z^{3k+2},xz+z^3)$  &  $(y,yz+z^6\pm z^{3k+2},z^3)$ & $H_{k+1}$ \\
$P_4^1$ & $(x,y,yz+z^7+z^8,xz+z^3)$  &  $(y,yz+z^7+z^{8},z^3)$ & $H_{3}$ \\
$P_4$ & $(x,y,yz+z^7,xz+z^3)$  &  $(x,z^7,xz+z^3)$ & not in the list \\
$Q_k$ & $(x,y,xz+z^4\pm y^kz^2+yz^2,z^3\pm y^kz)$  & $(y,z^4\pm y^kz^2+yz^2,z^3\pm y^kz)$
& not in the list \\
$R_k$ & $(x,y,xz+z^3,yz^2+z^4+z^{2k-1})$  &  $(x,xz+z^3,z^4+z^{2k-1})$ & $T_4$ \\
$S_{j,k}$ & $(x,y,xz+y^2z^2\pm z^{3j+2},z^3\pm y^{k}z)$  &  $(x,xz+z^{3j+2},z^3)$ & $H_{j+1}$ \\\hline 
\end{tabular} \end{table}

\section*{Acknowledgments}
This work is based on a part of the author's PhD thesis submitted at the University of Warwick in 2011 under the supervision of David Mond to whom she would like to express her deepest gratitude. Without his encouragement and guidance, this research would not have been commenced or improved.  The author would also like to thank Meral Tosun for helpful suggestions in the preparation of this manuscript and the referee for careful reading and constructive comments.

\bibliographystyle{amsplain}
\bibliography{a-reference}

\end{document}